\documentclass[11pt]{amsart}

\usepackage{amsmath}
\usepackage{amssymb}
\usepackage{latexsym}
\usepackage{url}

\author{David M. Ambrose}
\address{Department of Mathematics, Drexel University, Philadelphia, PA, USA}
\title[Mean field games with non-separable Hamiltonians]
{Strong solutions for time-dependent mean field games with non-separable Hamiltonians}

\newtheorem{theorem}{Theorem}

\begin{document}

\begin{abstract} We prove existence theorems for strong solutions of time-dependent 
mean field games with non-separable
Hamiltonian.
In a recent announcement, we showed existence of small, strong solutions for
mean field games with local coupling.  We first generalize that prior work  
to allow for non-separable Hamiltonians.  This proof is inspired by the work of Duchon and Robert on the 
existence of small-data vortex sheets in incompressible fluid mechanics. 
Our next existence result is in the case of weak coupling of the system; that is, we
allow the data to be of arbitrary size, but instead require that the (still possibly non-separable) Hamiltonian be small in a
certain sense.  The proof of this theorem relies upon an appeal to the implicit function theorem.
\end{abstract}

\maketitle

\section{Introduction}
We consider the following coupled system \cite{cardNotes}, known as the mean field games system:
\begin{equation}\label{uEquation}
u_{t}+\Delta u + \mathcal{H}(t,x,Du,m)=0,
\end{equation}
\begin{equation}\label{mEquation}
m_{t}-\Delta m  + \mathrm{div}(m\mathcal{H}_{p}(t,x,Du,m))=0.
\end{equation}
The independent variables $(t,x)$ are taken from $[0,T]\times\mathbb{T}^{n},$ for some $T>0$ and $n\geq1.$
The function $\mathcal{H}$ is known as the Hamiltonian, $m$ is a probability measure, and $u$ is a value function;
the interpretation of $u$ depends on the particular application.  The notation $\mathcal{H}_{p}$ denotes
$\displaystyle\frac{\partial}{\partial p}\mathcal{H}(t,x,p,m).$  This will be supplemented by the (temporal) boundary 
conditions
\begin{equation}\label{BCs}
m(0,x)=m_{0}(x),\qquad u(T,x)=G(m(T,\cdot)).
\end{equation}
The operator $G$ is known as the payoff function.

Mean field games have been introduced by Lasry and Lions \cite{lasry1}, \cite{lasry2}, 
\cite{lasry3} as a limiting approximation for problems from game theory with a large number of agents.
Two recent surveys of the theory and applications of mean field games systems are \cite{gomesSurvey}
and \cite{meanField2}.
In all prior existence and uniqueness theorems for time-dependent mean field games
which are known to the author, the Hamiltonian, $\mathcal{H},$ is assumed to be separable.
This means that it is assumed that there exist $H$ and $F$ such that $\mathcal{H}(t,x,p,m)=H(t,x,p)+F(t,x,m).$
In such a case, the function $H$ is still known as the Hamiltonian, but $F$ is then referred to as the \emph{coupling}.
However, in applications, non-separable Hamiltonians are frequently of interest \cite{mollNotes}.  
We do not assume separability in the present work.

In the case of separable Hamiltonian and local coupling, a number of works  prove the existence of weak solutions
\cite{porrettaCRAS}, \cite{porrettaFull}, \cite{porrettaARMA}.  This includes works in which the diffusion terms
either can be degenerate \cite{cardESAIM} or absent altogether \cite{cardNoDEA}.
The only works of which the author is aware which treat non-separable Hamiltonians are 
 \cite{gomesARXIV} and \cite{gomesExtended}; in these, weak solutions are proven to exist for
 stationary problems.

When the Hamiltonian is separable and when the coupling is a nonlocal smoothing operator, strong solutions have
been shown to exist \cite{lasry3}.  In the case of local coupling, strong solutions have been proven to
exist in some cases.  When the Hamiltonian is quadratic, $H=|Du|^{2},$ it is possible to use the Hopf-Cole transformation
to find a system more amenable to analysis \cite{gueantMMMAS}.  Under a number of technical assumptions, Gomes,
Pimentel, and S\'{a}nchez-Morgado have shown that strong solutions exist in the case of subquadratic or 
superquadratic Hamiltonians in the papers \cite{gomesSub} and \cite{gomesSuper}, respectively.  Gomes and Pimentel
have also treated logarithmic nonlinearities \cite{gomesLog}.  

The present work is an extension of the announcement \cite{ambroseMFG}.  In \cite{ambroseMFG}, the author
presented an existence and local uniqueness theorem for strong solutions of the system 
\eqref{uEquation}-\eqref{mEquation} in the
case of a separable Hamiltonian, with local coupling.  This theorem used function spaces based upon the Wiener
algebra, and was based upon work of Duchon and Robert for vortex sheets in incompressible, inviscid fluids 
\cite{duchonRobert},
and also upon the extension of \cite{duchonRobert} to finite time intervals by Milgrom and the author 
\cite{milgromAmbrose}.  
As long as certain mapping properties and Lipschitz estimates were satisfied by the Hamiltonian,
the coupling, and the payoff function, we were able to conclude that for initial measures $m_{0}$ sufficiently close
to the uniform measure on $\mathbb{T}^{n},$ a solution of the problem \eqref{uEquation}, \eqref{mEquation},
\eqref{BCs} exists and is unique in a ball about the origin.  The theorem of \cite{ambroseMFG} is complementary
to the results of \cite{gomesSuper} and \cite{gomesSub}, in that the assumptions are simpler, at the expense that
we only find small solutions.
In the present work, we include more details of the proof of \cite{ambroseMFG}, and we also
extend to the case of non-separable Hamiltonians.  

Going further beyond \cite{ambroseMFG}, we provide another existence theorem, still allowing for
non-separable Hamiltonians.  This theorem still requires smallness in some way, and we impose this by
placing a small parameter in front of the Hamiltonian.  Then, for data of arbitrary size, we use the implicit 
function theorem to conclude that solutions to the mean field games system exist for some interval of values
of the small parameter.

The remainder of this paper is organized as follows: In Section 2, we discuss our viewpoint that solvability of the
mean field games system is related to ellipticity of the system.  In Section 3, we generalize the results of 
\cite{ambroseMFG}, proving our first two main theorems on the existence of small, strong solutions for mean field games
with non-separable Hamiltonian (one of these is for payoff boundary conditions, and the other is for planning boundary
conditions).  In Section 4 we present our third main theorem, removing the smallness condition
on the data but placing a smallness condition on the Hamiltonian instead.

\section{A remark on ellipticity}\label{section:ellipticity}

The equations \eqref{uEquation}, \eqref{mEquation} are evidently a coupled system of equations, one of which is
forward parabolic and one of which is backwards parabolic.  While it has been remarked sometimes that there is
thus some elliptic behavior of such a coupled system, we take the point of view that \eqref{uEquation}, 
\eqref{mEquation} can be viewed as a quaslinear elliptic system.  We illustrate this with the following
calculation.

Let $\phi,$ $\psi$ satisfy the following coupled system:
\begin{equation}\label{modelSystem}
\phi_{t}=-\Delta\phi + a(\phi,\psi),\qquad \psi_{t}=\Delta\psi + b(\phi,\psi).
\end{equation}
We define $\Phi=\phi+\psi$ and $\Psi=\phi-\psi.$
We can then form the following system by adding and subtracting in \eqref{modelSystem}:
\begin{equation}\label{modelSystem2}
\Phi_{t}=-\Delta\Psi+\tilde{a}(\Phi,\Psi),\qquad \Psi_{t}=-\Delta\Phi+\tilde{b}(\Phi,\Psi).
\end{equation}

We take the time derivative of both equations in \eqref{modelSystem2}:
\begin{equation}\nonumber
\Phi_{tt}=\Delta^{2}\Phi-\Delta\tilde{b}(\Phi,\Psi)+(\tilde{a}(\Phi,\Psi))_{t},
\qquad \Psi_{tt}=\Delta^{2}\Psi-\Delta\tilde{a}(\Phi,\Psi)+(\tilde{b}(\Phi,\Psi))_{t}.
\end{equation}
If we had $\tilde{a}=\tilde{b}=0,$ then these are clearly elliptic equations in space-time.  If instead $\tilde{a}$ and 
$\tilde{b}$ depended on $\Phi,$ $\Psi,$ and their first spatial derivatives, we would have a semilinear elliptic system 
(the contributions of terms like $\Delta\tilde{b}$ or $\partial_{t}\tilde{b}$ would include up to third spatial 
derivatives of $\Phi$ and $\Psi,$ which are therefore lower-order terms).  However, in the mean field games case,
$\tilde{a}$ and $\tilde{b}$ include second spatial derivatives, because of the divergence term in \eqref{mEquation}.
Thus, if the mean field games system is elliptic, it is a quasilinear elliptic system at best.

We remark that we ensure ellipticity (and solvability) 
of the system \eqref{uEquation}, \eqref{mEquation} through our smallness assumptions.  In the first set of
results we present, the smallness assumption is on the data; this is the content of Section \ref{generalTheoremSection}.  
In our second kind of result, presented in Section \ref{largeDataSection}, 
we allow for general
data but instead consider the entire Hamiltonian to be small (i.e., we place a small parameter in front of $\mathcal{H}$).
We have no need in our
proofs for an assumption of convexity of the Hamiltonian.  In the absence of our smallness assumptions, however,
a convexity assumption could guarantee ellipticity as well.  Generically, however, without any special assumptions,
there would be a possibility of the 
mean field games system being of mixed  type.

We note that this formulation in terms of $\Phi$ and $\Psi$ was for illustrative purposes only, and it will not be used
in the sequel.
\section{The small data theorem}\label{generalTheoremSection}

In this section we will state and prove our first main theorem; this takes place in Section \ref{generalTheorem} below.  First,
we reformulate the problem in Section \ref{section:Duhamel}, we define our function spaces in Section \ref{section:Wiener},
and we establish estimates for relevant operators in Section \ref{section:operatorEstimates}.  We  
give an existence theorem for the planning problem in Section \ref{section:planning}, and we give
examples of Hamiltonians which satisfy our hypotheses in Section \ref{exampleSection}.

\subsection{Duhamel Formulation}\label{section:Duhamel}

We will be writing a Duhamel formula which requires integrating forward in time
from $t=0$ and which also requires integrating backwards in time from $t=T.$  It will therefore be helpful to introduce
the operators $I^{+}$ and $I^{-},$ defined as follows:
\begin{equation}\nonumber
(I^{+}f)(t,\cdot)=\int_{0}^{t}e^{\Delta(t-s)}f(s,\cdot)\ ds,
\end{equation}
\begin{equation}\nonumber
(I^{-}f)(t,\cdot)=\int_{t}^{T}e^{\Delta(s-t)}f(s,\cdot)\ ds,
\end{equation}
We also introduce the projection operator $\mathbb{P},$ which removes the mean of a periodic function.
That is, $\mathbb{P}f=f-\frac{1}{\mathrm{vol}(\mathbb{T}^{n})}\int_{\mathbb{T}^{n}}f.$
We define the constant $\bar{m}=1/\mathrm{vol}(\mathbb{T}^{n}).$

We change variables from $(u,m)$ to $(w,\mu),$ where $w=\mathbb{P}u$ and $\mu=\mathbb{P}m=m-\bar{m}.$
Notice that $Du=Dw.$
We define $\Xi(t,x,Dw,\mu)=\mathbb{P}\mathcal{H}(t,x,Du,m)$ and $\Theta(t,x,Dw,\mu)=\mathcal{H}_{p}(t,x,Du,m).$
The evolution equation and initial condition for $\mu$ are then as follows:
\begin{equation}\label{muEvolution}
\mu_{t}-\Delta\mu+\mathrm{div}(\mu\Theta(\cdot,\cdot,Dw,\mu))+\bar{m}\mathrm{div}(\Theta(\cdot,\cdot,Dw,\mu))=0,
\end{equation}
\begin{equation}\label{muBC}
\mu(0,x)=\mu_{0}(x):=m_{0}(x)-\bar{m}.
\end{equation}
For $t\in[0,T],$ we integrate forward from time zero, finding the following Duhamel formula for $\mu:$
\begin{equation}\label{duhamelMu}
\mu(t,\cdot)=e^{\Delta t}\mu_{0}
+I^{+}(\mathrm{div}(\mu \Theta(\cdot,\cdot,Dw,\mu)))(t,\cdot)
+\bar{m}(I^{+}(\mathrm{div}(\Theta(\cdot,\cdot,Dw,\mu))))(t,\cdot).
\end{equation}
We introduce the operator $I_{T},$ which is simply $I^{+}$ evaluated at time $t=T:$
$$I_{T}f=(I^{+}f)(T,\cdot).$$
Considering \eqref{duhamelMu} when $t=T,$ we introduce more notation:
\begin{equation}\label{defn:Amuw}
\mu(T,\cdot)=A(\mu,w):=e^{\Delta T}\mu_{0}+I_{T}(\mathrm{div}(\mu\Theta(\cdot,\cdot,Dw,\mu)))
+\bar{m}I_{T}(\mathrm{div}(\Theta(\cdot,\cdot,Dw,\mu))).
\end{equation}

The evolution equation and boundary condition for $w$ are
\begin{equation}\label{wEvolution}
w_{t}+\Delta w+\Xi(t,x,Dw,\mu)=0,
\end{equation}
\begin{equation}\label{wBC}
w(T,x)=\mathbb{P}G(x,m(T,\cdot))=\mathbb{P}G(x,\mu(T,\cdot)+\bar{m})=:\tilde{G}(\mu(T,\cdot)).
\end{equation}
Before writing the Duhamel formula for $w,$ we decompose $\Xi$ somewhat, to remove a linear term.
We write 
\begin{equation}\nonumber
\Xi(t,x,Dw,\mu)=b(t,x)\mu+\Upsilon(t,x,Dw,\mu),
\end{equation}
where the conditions to be satisfied by $b$ and $\Upsilon$ will be given below.  
We can now give our Duhamel formula for $w,$ which integrates backwards from time $T:$
\begin{equation}\label{duhamelWFirst}
w(t,\cdot)=
e^{\Delta(T-t)}\tilde{G}(A(\mu,w))-I^{-}(\mathbb{P}\Upsilon(\cdot,\cdot,Dw,\mu))(t)
-I^{-}(\mathbb{P}(b\mu))(t).
\end{equation}
We are not finished with our Duhamel formula, though; to eliminate the linear term (the final term on the right-hand side
of \eqref{duhamelWFirst}), we substitute from \eqref{duhamelMu}.  We arrive at the following:
\begin{multline}\label{duhamelWLast}
w(t,\cdot)=e^{\Delta(T-t)}\tilde{G}(A(\mu,w))-I^{-}(\mathbb{P}\Upsilon(\cdot,\cdot,Dw,\mu))(t)
-I^{-}(\mathbb{P}(be^{\Delta \cdot}\mu_{0}))(t)\\
-I^{-}(\mathbb{P}(bI^{+}\mathrm{div}(\mu\Theta(\cdot,\cdot,Dw,\mu))(\cdot)))(t)
-\bar{m}I^{-}(\mathbb{P}(bI^{+}\mathrm{div}(\Theta(\cdot,\cdot,Dw,\mu))(\cdot)))(t).
\end{multline}

We define a mapping, $\mathcal{T}(w,\mu)=(\mathcal{T}_{1}(w,\mu),\mathcal{T}_{2}(w,\mu)),$ based on the above.
The operators $\mathcal{T}_{1}$ and $\mathcal{T}_{2}$ are given by the right-hand sides of \eqref{duhamelMu}
and \eqref{duhamelWLast}.  That is, we define
\begin{equation}\label{defn:T1}
\mathcal{T}_{1}(w,\mu)=e^{\Delta t}\mu_{0}
+(I^{+}\mathrm{div}(\mu \Theta(\cdot,\cdot,Dw,\mu)))(t,\cdot)
+\bar{m}(I^{+}(\mathrm{div}(\Theta(\cdot,\cdot,Dw,\mu))))(t,\cdot),
\end{equation}
\begin{multline}\label{defn:T2}
\mathcal{T}_{2}(w,\mu)=e^{\Delta(T-t)}\tilde{G}(A(\mu,w))-I^{-}(\mathbb{P}\Upsilon(\cdot,\cdot,Dw,\mu))(t)
-I^{-}(\mathbb{P}(be^{\Delta \cdot}\mu_{0}))(t)\\
-I^{-}(\mathbb{P}(bI^{+}\mathrm{div}(\mu\Theta(\cdot,\cdot,Dw,\mu))(\cdot)))(t)
-\bar{m}I^{-}(\mathbb{P}(bI^{+}\mathrm{div}(\Theta(\cdot,\cdot,Dw,\mu))(\cdot)))(t).
\end{multline}
In Section \ref{generalTheorem} below, we will prove the existence of  
fixed points of $\mathcal{T},$ and thus solutions of our system.
Before finding fixed points, however, we must discuss function spaces.

\subsection{Function spaces}\label{section:Wiener}
As in \cite{ambroseMFG}, we let $T>0$ and $\alpha\in\left(0,\frac{T}{2}\right)$ be given, and
we define $\beta:[0,T]\rightarrow[0,\alpha]$ by
\begin{equation}\nonumber
\beta(s)=\left\{\begin{array}{ll}
2\alpha s/T,& s\in[0,T/2],\\
2\alpha-2\alpha s/T,& s\in[T/2, T].
\end{array}\right.
\end{equation}
We will study solutions of the mean field games system in function spaces based on the Wiener
algebra (that is, spaces based on the idea of taking Fourier coefficients to be in $\ell^{1}$).
Our spaces will use the exponential weight $e^{\beta(t)|k|},$ and this will allow us to conclude analyticity of solutions.  

We define $B^{j}$ to be the set of continuous functions from $\mathbb{T}^{n}$ to $\mathbb{R}$ such that
for all $f\in B^{j},$ the norm $|f|_{B^{j}}$ is finite, with
\begin{equation}\nonumber
|f|_{B^{j}}=\sum_{k\in\mathbb{Z}^{n}}(1+|k|^{j})|\hat{f}(k)|.
\end{equation}
We extend this to a space-time version $\mathcal{B}_{\alpha}^{j};$ this is the set of all functions in $C([0,T];B^{j})$ such
that $\|f\|_{\mathcal{B}_{\alpha}^{j}}$ is finite, with
\begin{equation}
\|f\|_{\mathcal{B}_{\alpha}^{j}}=\sum_{k\in\mathbb{Z}^{n}}\sup_{t\in[0,T]}
(1+|k|^{j})e^{\beta(t)|k|}|\hat{f}(t,k)|.
\end{equation}

We have an algebra property for $\mathcal{B}_{\alpha}^{j}:$ if $\|f\|_{\mathcal{B}_{\alpha}^{j}}$ and
$\|g\|_{\mathcal{B}_{\alpha}^{j}}$ are both finite, then so is $\|fg\|_{\mathcal{B}_{\alpha}^{j}}.$
We demonstrate this for $j=0;$ to begin, we write the definition:
\begin{equation}\nonumber
\|fg\|_{\mathcal{B}_{\alpha}^{0}}=2\sum_{k\in\mathbb{Z}^{n}}\sup_{t\in[0,T]}e^{\beta(t)|k|}
\left|\widehat{fg}(k)\right|
\end{equation}
Next, we use the convolution formula for $\widehat{fg}$ and the triangle inequality:
\begin{equation}\nonumber
\|fg\|_{\mathcal{B}_{\alpha}^{0}}\leq 2\sum_{k\in\mathbb{Z}^{n}}
\sup_{t\in[0,T]}\sum_{j\in\mathbb{Z}^{n}}
\left(e^{\beta(t)|k-j|}|\hat{f}(t,k-j)|\right)
\left(e^{\beta(t)|j|}|\hat{g}(t,j)|\right).
\end{equation}
We next pass the supremum through the second summation, and we use Tonelli's theorem:
\begin{equation}\label{basicAlgebraProperty}
\|fg\|_{\mathcal{B}_{\alpha}^{0}}\leq2\sum_{j\in\mathbb{Z}^{n}}\sum_{k\in\mathbb{Z}^{n}}
\left(\sup_{t\in[0,T]}e^{\beta(t)|k-j|}|\hat{f}(t,k-j)|\right)
\left(\sup_{t\in[0,T]}e^{\beta(t)|j|}|\hat{g}(t,j)|\right)
\leq \|f\|_{\mathcal{B}_{\alpha}^{0}}\|g\|_{\mathcal{B}_{\alpha}^{0}}.
\end{equation}

Finally, the product rule and \eqref{basicAlgebraProperty} imply the algebra property for $j\in\mathbb{N}:$ for any
$j\in\mathbb{N},$ there exists $c_{j}>0$ such that for all 
$f\in\mathcal{B}_{\alpha}^{j}$ and for all $g\in\mathcal{B}_{\alpha}^{j},$  we have $fg\in\mathcal{B}_{\alpha}^{j}$ with the
estimate
\begin{equation}\nonumber
\|fg\|_{\mathcal{B}_{\alpha}^{j}}\leq c_{j}\|f\|_{\mathcal{B}_{\alpha}^{j}}\|g\|_{\mathcal{B}_{\alpha}^{j}}.
\end{equation}

\subsection{Operator estimates}\label{section:operatorEstimates}
For any $j\in\mathbb{N},$ we show in this section that
$I^{+}$ and $I^{-}$ are bounded linear operators from $\mathcal{B}_{\alpha}^{j}$ to $\mathcal{B}_{\alpha}^{j+2}.$
Technically, we consider the subspace of $\mathcal{B}_{\alpha}^{j}$ consisting of mean-zero functions.

Let $j\in\mathbb{N},$ and let $h\in\mathcal{B}_{\alpha}^{j}$ have zero spatial mean at each time.
We begin with the simple statement of the norm of $I^{+}h:$
\begin{equation}\nonumber
\|I^{+}h\|_{\mathcal{B}_{\alpha}^{j+2}}=\sum_{k\in\mathbb{Z}^{n}\setminus\{0\}}(1+|k|^{j+2})\sup_{t\in[0,T]}
e^{\beta(t)|k|}\left|\int_{0}^{t}e^{-|k|^{2}(t-s)}\hat{h}(s,k)\ ds\right|.
\end{equation}
We replace $1+|k|^{j+2}$ with $2|k|^{j+2},$ for simplicity.
We multiply and divide with the exponentials, and we use some elementary inequalities:
\begin{multline}\nonumber
\|I^{+}h\|_{\mathcal{B}_{\alpha}^{j+2}}\leq2\sum_{k\in\mathbb{Z}^{n}\setminus\{0\}}|k|^{j+2}\sup_{t\in[0,T]}
e^{\beta(t)|k|}\left|\int_{0}^{t}e^{-|k|^{2}(t-s)}e^{-\beta(s)|k|}e^{\beta(s)|k|}\hat{h}(s,k)\ ds\right|
\\
\leq2\sum_{k\in\mathbb{Z}^{n}\setminus\{0\}}|k|^{j+2}\sup_{t\in[0,T]}\left[e^{\beta(t)|k|}\int_{0}^{t}
e^{-|k|^{2}(t-s)}e^{-\beta(s)|k|}\sup_{\tau\in[0,T]}\left[e^{\beta(\tau)|k|}|\hat{h}(\tau,k)|
\right]\ ds\right].
\end{multline}
Next, we pull the final supremum through the integral and through the first supremum, and we rearrange the factors of $|k|:$
\begin{equation}\nonumber
\|I^{+}h\|_{\mathcal{B}_{\alpha}^{j+2}}\leq2\sum_{k\in\mathbb{Z}^{n}\setminus\{0\}}
\left[|k|^{j}\sup_{\tau\in[0,T]}e^{\beta(\tau)|k|}|\hat{h}(\tau,k)|\right]
\left[\sup_{t\in[0,T]}|k|^{2}e^{\beta(t)|k|}\int_{0}^{t}e^{-|k|^{2}(t-s)-\beta(s)|k|}\ ds
\right].
\end{equation}
For the second quantity in square brackets, we now take its supremum over values of $k,$ and pull this through the sum.
This yields the following:
\begin{equation}\nonumber
\|I^{+}h\|_{\mathcal{B}_{\alpha}^{j+2}}\leq2\left[
\sup_{t\in[0,T]}\sup_{k\in\mathbb{Z}^{n}\setminus\{0\}}|k|^{2}e^{\beta(t)|k|}\int_{0}^{t}
e^{-|k|^{2}(t-s)-\beta(s)|k|}\ ds\right]
\|h\|_{\mathcal{B}_{\alpha}^{j}}.
\end{equation}
Thus, to verify that $I^{+}$ is indeed a bounded linear operator between $\mathcal{B}_{\alpha}^{j}$ and
$\mathcal{B}_{\alpha}^{j+2}$ as claimed, we need only to verify that the quantity
\begin{equation}\label{quantity:toBeBounded}
\sup_{t\in[0,T]}\sup_{k\in\mathbb{Z}^{n}\setminus\{0\}}|k|^{2}e^{\beta(t)|k|-t|k|^{2}}\int_{0}^{t}
e^{|k|^{2}s-\beta(s)|k|}\ ds
\end{equation}
is finite.

We will check that the quantity \eqref{quantity:toBeBounded} is finite by first checking values $t\in[0,T/2].$
For $t\in[0,T/2],$ we need only consider $s\in[0,T/2]$ as well, and thus $\beta(s)=2\alpha s/T.$  We may thus
compute the integral as follows:
\begin{equation}\label{integral:computed1}
\int_{0}^{t}e^{|k|^{2}s-\beta(s)|k|}\ ds = \int_{0}^{t}e^{|k|^{2}s-2\alpha s|k|/T}\ ds
=\frac{\exp\{|k|^{2}t-2\alpha t|k|/T\}-1}{|k|^{2}-2\alpha|k|/T}.
\end{equation}
Using \eqref{integral:computed1} with \eqref{quantity:toBeBounded}, we have the following:
\begin{multline}\label{bounded:case1}
\sup_{t\in[0,T/2]}\sup_{k\in\mathbb{Z}^{n}\setminus\{0\}}|k|^{2}\exp\{2\alpha t|k|/T-t|k|^{2}\}\cdot
\frac{\exp\{|k|^{2}t-2\alpha t|k|/T\}-1}{|k|^{2}-2\alpha|k|/T}
\\
= \sup_{t\in[0,T/2]}\sup_{k\in\mathbb{Z}^{n}\setminus\{0\}}\frac{1-\exp\{2\alpha t|k|/T-t|k|^{2}\}}{1-2\alpha/(T|k|)}
\leq \sup_{k\in\mathbb{Z}^{n}\setminus\{0\}}\frac{1}{1-2\alpha/(T|k|)}=\frac{1}{1-2\alpha/T}=\frac{T}{T-2\alpha}.
\end{multline}
(Recall that we have assumed $\alpha\in(0,T/2).$)

We now turn to the case $t\in[T/2,T].$  In this case, we may write the integral from \eqref{quantity:toBeBounded} as
follows:
\begin{equation}\label{integral:case2}
\int_{0}^{t}e^{|k|^{2}s-\beta(s)|k|}\ ds
=\int_{0}^{T/2}e^{|k|^{2}s-\beta(s)|k|}\ ds
+\int_{T/2}^{t}e^{|k|^{2}s-\beta(s)|k|}\ ds.
\end{equation}
The first integral on the right-hand side of \eqref{integral:case2} was previously computed; specifically, we have its value
by setting $t=T/2$ in \eqref{integral:computed1}.  For the second integral on the right-hand side of \eqref{integral:case2},
we consider only $s\in[T/2,T],$ and thus $\beta(s)=2\alpha-2\alpha s/T:$
\begin{multline}\label{integral:computed2}
\int_{T/2}^{t}e^{|k|^{2}s-\beta(s)|k|}\ ds=\int_{T/2}^{t}e^{|k|^{2}s-2\alpha|k|+2\alpha s|k|/T}\ ds\\
=e^{-2\alpha|k|}\left(\frac{\exp\{|k|^{2}t+2\alpha t|k|/T\}-\exp\{|k|^{2}T/2+\alpha|k|\}}{|k|^{2}+2\alpha|k|/T}\right).
\end{multline}
Combining \eqref{integral:computed1} (setting $t=T/2$ there) with \eqref{integral:computed2}, we have found the following
for $t\in[T/2,T]:$
\begin{multline}\nonumber
\int_{0}^{t}e^{|k|^{2}s-\beta(s)|k|}\ ds\\
=\frac{\exp\{|k|^{2}T/2-\alpha|k|\}-1}{|k|^{2}-2\alpha|k|/T}
+e^{-2\alpha|k|}\left(\frac{\exp\{|k|^{2}t+2\alpha t|k|/T\}-\exp\{|k|^{2}T/2+\alpha|k|\}}{|k|^{2}+2\alpha|k|/T}\right).
\end{multline}
We use this with \eqref{quantity:toBeBounded}, keeping in mind that since $t\in[T/2,T],$ we have
$\beta(t)=2\alpha-2\alpha t/T:$
\begin{equation}\nonumber
\sup_{t\in[T/2,T]}\sup_{k\in\mathbb{Z}^{n}\setminus\{0\}}|k|^{2}e^{\beta(t)|k|-t|k|^{2}}\int_{0}^{t}
e^{|k|^{2}s-\beta(s)|k|}\ ds\leq I + II,
\end{equation}
where $I$ and $II$ are given by
\begin{equation}\nonumber
I=\sup_{t\in[T/2,T]}\sup_{k\in\mathbb{Z}^{n}\setminus\{0\}}|k|^{2}\exp\{2\alpha|k|-2\alpha|k|t/T-t|k|^{2}\}
\left(\frac{\exp\{|k|^{2}T/2-\alpha|k|\}-1}{|k|^{2}-2\alpha|k|/T}\right),
\end{equation}
\begin{equation}\nonumber
II=\sup_{t\in[T/2,T]}\sup_{k\in\mathbb{Z}^{n}\setminus\{0\}}|k|^{2}
\exp\{-2\alpha |k|t/T-t|k|^{2}\}
\left(\frac{\exp\{|k|^{2}t+2\alpha t|k|/T\}-\exp\{|k|^{2}T/2+\alpha|k|\}}{|k|^{2}+2\alpha|k|/T}\right).
\end{equation}

We now bound $I$ and $II.$  We start with $I:$
\begin{equation}\nonumber
I\leq\sup_{t\in[0,T/2]}\sup_{k\in\mathbb{Z}^{n}\setminus\{0\}}\frac{1}{1-2\alpha/(T|k|)}
\exp\left\{\alpha|k|-2\alpha|k|t/T-t|k|^{2}+|k|^{2}T/2\right\}.
\end{equation}
Using the condition $t\geq T/2,$ we may simply bound this as
\begin{equation}\label{bounded:I}
I\leq\frac{1}{1-2\alpha/T}=\frac{T}{T-2\alpha}.
\end{equation}
We similarly bound $II$ as follows:
\begin{equation}\label{bounded:II}
II\leq \sup_{k\in\mathbb{Z}^{n}\setminus\{0\}}\frac{1}{1+2\alpha/(T|k|)}=1.
\end{equation}

Based on all of the above (specifically considering \eqref{bounded:case1} 
and adding \eqref{bounded:I} and \eqref{bounded:II}),
we conclude that $I^{+}$ is a bounded linear operator, with operator norm satisfying
\begin{equation}\nonumber
\|I^{+}\|_{\mathcal{B}_{\alpha}^{j}\rightarrow\mathcal{B}_{\alpha}^{j+2}}\leq \frac{2T}{T-2\alpha}+2.
\end{equation}
We omit the details of the proof of the boundedness of $I^{-},$ because they are completely analogous to those for $I^{+}.$
In fact, we end with the exact same estimate of the operator norm:
\begin{equation}\nonumber
\|I^{-}\|_{\mathcal{B}_{\alpha}^{j}\rightarrow\mathcal{B}_{\alpha}^{j+2}}\leq \frac{2T}{T-2\alpha}+2.
\end{equation}

\subsection{The first main theorem}\label{generalTheorem}

We will show that the operator $\mathcal{T}$ is a local contraction, under certain assumptions on the Hamiltonian
$\mathcal{H}$ and the payoff function $G.$  In particular, these must satisfy certain Lipschitz properties.  Together
wtih the operator estimates of Section \ref{section:operatorEstimates}, 
these Lipschitz properties will allow us to prove Theorem \ref{smallDataTheorem1}.  Our assumption on the payoff
function, $G$ is {\bf(A1)} below, and our assumption on the Hamiltonian, $\mathcal{H},$ is {\bf(A2)} below.

\noindent
{\bf (A1)} $\tilde{G}(0)=0,$ and $\tilde{G}$ is Lipschitz in a neighborhood of the origin in $B^{2}.$
Specifically, the Lipschitz property we assume is that there exists $c>0$ and $\epsilon>0$ such that for all 
$a_{1},$ $a_{2}$ satisfying $|a_{1}|_{B^{2}}<\epsilon$ and $|a_{2}|_{B^{2}}<\epsilon,$
\begin{equation}\label{lipschitzG}
|\tilde{G}(a_{1})-\tilde{G}(a_{2})|_{B^{2}}\leq c |a_{1}-a_{2}|_{B^{2}}.
\end{equation}
For example, $G(a)=a$ or $G(a)=a^2$ certainly satisfy this assumption.

\noindent
{\bf (A2)} $\Upsilon(\cdot,\cdot,0,0)=0$ and $\Theta(\cdot,\cdot,0,0)=0.$  There exists a continuous function
$\Phi_{1}:\mathcal{B}_{\alpha}^{2}\times\mathcal{B}_{\alpha}^{2}\rightarrow\mathbb{R}$
such that as $(w_{1},w_{2},\mu_{1},\mu_{2})\rightarrow 0,$ we have $\Phi_{1}(w_{1},w_{2},\mu_{1},\mu_{2})\rightarrow 0,$
and such that
\begin{multline}\label{lipschitzHp}
\|\Theta(\cdot,\cdot,Dw_{1},\mu_{1})
-\Theta(\cdot,\cdot,Dw_{2},\mu_{2})\|_{(\mathcal{B}_{\alpha}^{1})^{n}}
\\
\leq
\Phi_{1}(w_{1},w_{2},\mu_{1},\mu_{2})
\left(\|Dw_{1}-Dw_{2}\|_{(\mathcal{B}_{\alpha}^{1})^{n}}+\|\mu_{1}-\mu_{2}\|_{\mathcal{B}_{\alpha}^{2}}\right).
\end{multline}
There exists a continuous function
$\Phi_{2}:\mathcal{B}_{\alpha}^{2}\times\mathcal{B}_{\alpha}^{2}\rightarrow\mathbb{R}$
such that as $(w_{1},w_{2},\mu_{1},\mu_{2})\rightarrow 0,$ we have $\Phi_{2}(w_{1},w_{2},\mu_{1},\mu_{2})\rightarrow 0,$
and such that
\begin{equation}\label{lipschitzH}
\|\mathbb{P}\Upsilon(\cdot,\cdot,Dw_{1},\mu_{1})
-\mathbb{P}\Upsilon(\cdot,\cdot,Dw_{2},\mu_{2})\|_{\mathcal{B}_{\alpha}^{0}}\leq
\Phi_{2}(w_{1},w_{2},\mu_{1},\mu_{2})
\left(\|Dw_{1}-Dw_{2}\|_{\mathcal{B}_{\alpha}^{1}}+\|\mu_{1}-\mu_{2}\|_{\mathcal{B}_{\alpha}^{2}}\right).
\end{equation}

Having stated our assumptions, we are now ready to state the first of our main theorems.

\begin{theorem}\label{smallDataTheorem1}
Let $T>0$ and $\alpha\in(0,T/2)$ be given.  Let assumptions {\bf(A1)} and {\bf(A2)} be satisfied,
and assume $b\in\mathcal{B}_{\alpha}^{0}.$ 
There exists $\delta>0$ such that if $m_{0}$ is a probability measure such that $\mu_{0}=m_{0}-\bar{m}$
satisfies $|\mu_{0}|_{B^{2}}<\delta,$ then the system \eqref{uEquation}, \eqref{mEquation}, 
\eqref{BCs} has a strong, locally unique 
solution $(u,m)\in(\mathcal{B}_{\alpha}^{2})^{2}.$  Furthermore, for all $t\in(0,T),$ each of $u(t,\cdot)$ and $m(t,\cdot)$ are 
analytic, and $m(t,\cdot)$ is a probability measure.
\end{theorem}

\begin{proof} We will prove that $\mathcal{T}$ is a contraction on a set $X.$
We will let $X$ be the closed ball in 
$\mathcal{B}_{\alpha}^{2}\times\mathcal{B}_{\alpha}^{2}$ centered at a point $(a_{0},b_{0})$ with radius $r_{*}.$
While $r_{*}$ will be determined soon, we define $(a_{0},b_{0})$ now to be 
\begin{equation}\label{definitionOfBallX}
(a_{0},b_{0})=(e^{\Delta t}\mu_{0}, e^{\Delta(T-t)}\tilde{G}(e^{\Delta T}\mu_{0})
-I^{-}(\mathbb{P}(be^{\Delta\cdot}\mu_{0}))(t)).
\end{equation}
We note that $(a_{0},b_{0})$ can be made small by taking $\mu_{0}$ small.

To show that $\mathcal{T}$ is a contraction, our first task is to demonstrate that $\mathcal{T}$ maps $X$ to $X.$
To this end, we let $(w,\mu)\in X.$  We first must conclude that 
$\mathcal{T}(w,\mu)\in\mathcal{B}_{\alpha}^{2}\times\mathcal{B}_{\alpha}^{2}.$
It is immediate that if $f\in B^{2},$ then each of $e^{\Delta t}f$ and $e^{\Delta(T-t)}f$ are 
in $\mathcal{B}_{\alpha}^{2}.$  This fact, together with the assumptions {\bf(A1)} and {\bf(A2)} and
the previously demonstrated mapping properties of $I^{+}$ and $I^{-},$ allows us to conclude that 
$\mathcal{T}(w,\mu)\in\mathcal{B}_{\alpha}^{2}\times\mathcal{B}_{\alpha}^{2},$ as desired.

Next, to continue the demonstration that $\mathcal{T}$ maps $X$ to $X,$ 
we will show that $\|\mathcal{T}_{1}(w,\mu)-a_{0}\|_{\mathcal{B}_{\alpha}^{2}}
\leq r_{*}/2,$ and we will show that $\|\mathcal{T}_{2}(w,\mu)-b_{0}\|_{\mathcal{B}_{\alpha}^{2}}\leq r_{*}/2.$ 
We begin with $\mathcal{T}_{1}.$  
Recalling the definition \eqref{defn:T1}, we see that in order 
to show $\|\mathcal{T}_{1}(w,\mu)-a_{0}\|_{\mathcal{B}_{\alpha}^{2}}
\leq r_{*}/2,$ it is sufficient to prove the following two inequalities:
\begin{equation}\label{xToX:1}
\|I^{+}\mathrm{div}(\mu\Theta(\cdot,\cdot,Dw,\mu))\|_{\mathcal{B}_{\alpha}^{2}}\leq \frac{r_{*}}{4},
\end{equation}
\begin{equation}\label{xToX:2}
\|\bar{m}I^{+}\mathrm{div}(\Theta(\cdot,\cdot,Dw,\mu))\|_{\mathcal{B}_{\alpha}^{2}}\leq \frac{r_{*}}{4}.
\end{equation}
We begin by establishing \eqref{xToX:1}.  From the boundedness properties of $I^{+}$ and the fact that the divergence is a 
first-order operator, and from the algebra properties of the $\mathcal{B}_{\alpha}^{j}$ we can immediately write
\begin{equation}\label{firstXtoXQuantity}
\|I^{+}\mathrm{div}(\mu\Theta(\cdot,\cdot,Dw,\mu))\|_{\mathcal{B}_{\alpha}^{2}}\leq
c\|\mu\Theta(\cdot,\cdot,Dw,\mu)\|_{\mathcal{B}_{\alpha}^{1}}
\leq c\|\mu\|_{\mathcal{B}_{\alpha}^{1}}\|\Theta(\cdot,\cdot,Dw,\mu)\|_{\mathcal{B}_{\alpha}^{1}}.
\end{equation}

From \eqref{lipschitzHp}, setting $w_{2}=\mu_{2}=0$ there, and since $\Phi_{1}$ is continuous and $X$ is bounded,
we see that there exists a constant such that
\begin{equation}\label{firstXtoXPTheta}
\|\Theta(\cdot,\cdot,Dw,\mu)\|_{(\mathcal{B}_{\alpha}^{1})^{n}}
\leq c (\|w\|_{\mathcal{B}_{\alpha}^{2}}+\|\mu\|_{\mathcal{B}_{\alpha}^{2}}).
\end{equation}
Since $w\in X$ and $\mu\in X,$ we have 
\begin{equation}\label{XtoXBall}
\|w\|_{\mathcal{B}_{\alpha}^{2}}+\|\mu\|_{\mathcal{B}_{\alpha}^{2}}
\leq \|a_{0}\|_{\mathcal{B}_{\alpha}^{2}}+\|b_{0}\|_{\mathcal{B}_{\alpha}^{2}}+r_{*}.
\end{equation}
Combining \eqref{firstXtoXQuantity}, \eqref{firstXtoXPTheta}, and \eqref{XtoXBall},
we have
\begin{equation}\label{firstXtoXAlmost}
\|I^{+}\mathrm{div}(\mu\Theta(\cdot,\cdot,Dw,\mu)\|_{\mathcal{B}_{\alpha}^{2}}\leq c
(\|a_{0}\|_{\mathcal{B}_{\alpha}^{2}}+\|b_{0}\|_{\mathcal{B}_{\alpha}^{2}}+r_{*})^{2}.
\end{equation}
Thus, we see that if  $r_{*}$ is chosen small enough so that
\begin{equation}\nonumber
4cr_{*}^{2}\leq \frac{r_{*}}{4},
\end{equation}
and if $\mu_{0}$ is chosen small enough so that
\begin{equation}\nonumber
\|a_{0}\|_{\mathcal{B}_{\alpha}^{2}}+\|b_{0}\|_{\mathcal{B}_{\alpha}^{2}}\leq r_{*},
\end{equation}
then \eqref{firstXtoXAlmost} implies \eqref{xToX:1}.

We next demonstrate \eqref{xToX:2}.  
Similarly to the previous term, we have the following bounds:
\begin{equation}\nonumber
\|\bar{m}I^{+}\mathrm{div}\Theta(\cdot,\cdot,Dw,\mu)\|_{\mathcal{B}_{\alpha}^{2}}
\leq \bar{m}\|\Theta(\cdot,\cdot,Dw,\mu)\|_{(\mathcal{B}_{\alpha}^{1})^{n}}
\leq c\bar{m}\Phi_{1}(w,0,\mu,0)\left(\|w\|_{\mathcal{B}_{\alpha}^{2}}+\|\mu\|_{\mathcal{B}_{\alpha}^{2}}\right).
\end{equation}
We again bound the norm of $(w,\mu)$ in terms of $a_{0},$ $b_{0},$ and $r_{*},$ and we also bound $\Phi_{1}$ with
its maximum value on $X:$
\begin{equation}\label{xToX:2:almost}
\|\bar{m}I^{+}\mathrm{div}\mathbb{P}\Theta(\cdot,\cdot,Dw,\mu)\|_{\mathcal{B}_{\alpha}^{2}}
\leq c\bar{m}\left(\max_{(w,\mu)\in X}\Phi_{1}(w,0,\mu,0)\right)
\left(\|a_{0}\|_{\mathcal{B}_{\alpha}^{2}}+\|b_{0}\|_{\mathcal{B}_{\alpha}^{2}}+r_{*}\right).
\end{equation}
We again require that $a_{0}$ and $b_{0}$ are small enough so that
\begin{equation}\nonumber
\|a_{0}\|_{\mathcal{B}_{\alpha}^{2}}+\|b_{0}\|_{\mathcal{B}_{\alpha}^{2}}\leq r_{*}.
\end{equation}
Since we know $\Phi_{1}$ is continuous and since $\Phi_{1}(0,0,0,0)=0,$ we may take $\mu_{0}$ and 
$r_{*}$ small enough so that 
\begin{equation}\nonumber
\max_{(w,\mu)\in X}\Phi_{1}(w,0,\mu,0)\leq \frac{1}{8c\bar{m}}.
\end{equation}
With these conditions, we see that \eqref{xToX:2:almost} implies \eqref{xToX:2}.

Next, we consider $\mathcal{T}_{2}.$  Keeping in mind the definition \eqref{defn:T2} as well as the definition of
$b_{0},$ we see that it is sufficient to demonstrate the following four bounds:
\begin{equation}\label{t2:XtoX:1}
\|e^{\Delta(T-t)}\tilde{G}(e^{\Delta T}\mu_{0})-e^{\Delta(T-t)}\tilde{G}(A(\mu,w))\|_{\mathcal{B}_{\alpha}^{2}}
\leq \frac{r_{*}}{4},
\end{equation}
\begin{equation}\label{t2:XtoX:2}
\|I^{-}(\mathbb{P}\Upsilon(\cdot,\cdot,Dw,\mu))\|_{\mathcal{B}_{\alpha}^{2}}\leq\frac{r_{*}}{4},
\end{equation}
\begin{equation}\label{t2:XtoX:3}
\|I^{-}(\mathbb{P}(bI^{+}(\mathrm{div}(\mu\Theta(\cdot,\cdot,Dw,\mu)))))\|_{\mathcal{B}_{\alpha}^{2}}\leq\frac{r_{*}}{4},
\end{equation}
\begin{equation}\label{t2:XtoX:4}
\|\bar{m}I^{-}(\mathbb{P}(bI^{+}(\mathrm{div}(\Theta(\cdot,\cdot,Dw,\mu)))))\|_{\mathcal{B}_{\alpha}^{2}}\leq\frac{r_{*}}{4}.
\end{equation}
The details of the proof of estimates \eqref{t2:XtoX:2}, \eqref{t2:XtoX:3}, and \eqref{t2:XtoX:4} are very similar to the 
details already demonstrated for $\mathcal{T}_{1},$ so we will omit these.

We will now demonstrate \eqref{t2:XtoX:1}.
It is immediate that
\begin{equation}\nonumber
\|e^{\Delta(T-t)}\tilde{G}(e^{\Delta T}\mu_{0})-e^{\Delta(T-t)}\tilde{G}(A(\mu,w))\|_{\mathcal{B}_{\alpha}^{2}}
\leq c|\tilde{G}(e^{\Delta T}\mu_{0})-\tilde{G}(A((\mu,w))|_{B^{2}}.
\end{equation}
We next use the Lipschitz estimate \eqref{lipschitzG} as well as the definition of $A,$ \eqref{defn:Amuw},
and the triangle inequality,
finding the following estimate:
\begin{multline}\label{almostXtoX}
\|e^{\Delta(T-t)}\tilde{G}(e^{\Delta T}\mu_{0})-e^{\Delta(T-t)}\tilde{G}(A(\mu,w))\|_{\mathcal{B}_{\alpha}^{2}}
\leq c|I_{T}(\mathrm{div}(\mu\Theta(\cdot,\cdot,Dw,\mu)))|_{B^{2}}\\
+c\bar{m}|I_{T}(\mathrm{div}(\Theta(\cdot,\cdot,Dw,\mu)))|_{B^{2}}.
\end{multline}
Next, we may bound the right-hand side of \eqref{almostXtoX} using $I^{+}$ rather than $I_{T}:$
\begin{multline}\nonumber
\|e^{\Delta(T-t)}\tilde{G}(e^{\Delta T}\mu_{0})-e^{\Delta(T-t)}\tilde{G}(A(\mu,w))\|_{\mathcal{B}_{\alpha}^{2}}
\leq c\|I^{+}(\mathrm{div}(\mu\Theta(\cdot,\cdot,Dw,\mu)))\|_{\mathcal{B}_{\alpha}^{2}}\\
+c\bar{m}\|I^{+}(\mathrm{div}(\Theta(\cdot,\cdot,Dw,\mu)))\|_{\mathcal{B}_{\alpha}^{2}}.
\end{multline}
From this point, the details are similar to those provided above for $\mathcal{T}_{1}.$  We have shown that
it is possible to choose $r_{*}$ and $\mu_{0}$ appropriately so that $\mathcal{T}$ does indeed map $X$ to $X.$

We are now ready to demonstrate the contraction estimate.  We will show that if $\mu_{0}$ is taken sufficiently small,
then there exists $\lambda\in(0,1)$ such that for all $(w_{1},\mu_{1})\in X$ and $(w_{2},\mu_{2})\in X,$ we have
\begin{equation}\label{contraction:toShow}
\|\mathcal{T}(w_{1},\mu_{1})-\mathcal{T}(w_{2},\mu_{2})\|_{\mathcal{B}_{\alpha}^{2}\times\mathcal{B}_{\alpha}^{2}}
\leq \lambda\left(\|w_{1}-w_{2}\|_{\mathcal{B}_{\alpha}^{2}}+\|\mu_{1}-\mu_{2}\|_{\mathcal{B}_{\alpha}^{2}}\right).
\end{equation}
Considering the definitions \eqref{defn:T1} and \eqref{defn:T2}, we see that repeated use of the triangle inequality 
implies that it is sufficient to establish the following bounds:
\begin{equation}\label{contraction:1}
\left\|I^{+}\left(\mathrm{div}\Big(\mu_{1}\Theta(\cdot,\cdot,Dw_{1},\mu_{1})-\mu_{2}\Theta(\cdot,\cdot,Dw_{2},\mu_{2})
\Big)\right)\right\|_{\mathcal{B}_{\alpha}^{2}}\leq \frac{1}{7}\left(\|w_{1}-w_{2}\|_{\mathcal{B}_{\alpha}^{2}}
+\|\mu_{1}-\mu_{2}\|_{\mathcal{B}_{\alpha}^{2}}\right),
\end{equation}
\begin{equation}\label{contraction:2}
\bar{m}\left\|I^{+}\left(\mathrm{div}\Big(\Theta(\cdot,\cdot,Dw_{1},\mu_{1})-\Theta(\cdot,\cdot,Dw_{2},\mu_{2})
\Big)\right)\right\|_{\mathcal{B}_{\alpha}^{2}}\leq \frac{1}{7}\left(\|w_{1}-w_{2}\|_{\mathcal{B}_{\alpha}^{2}}
+\|\mu_{1}-\mu_{2}\|_{\mathcal{B}_{\alpha}^{2}}\right),
\end{equation}
\begin{equation}\label{contraction:3}
\left\|e^{\Delta(T-t)}\Big(\tilde{G}(A(\mu_{1},w_{1}))-\tilde{G}(A(\mu_{2},w_{2}))\Big)\right\|_{\mathcal{B}_{\alpha}^{2}}
\leq\frac{1}{7}\left(\|w_{1}-w_{2}\|_{\mathcal{B}_{\alpha}^{2}}
+\|\mu_{1}-\mu_{2}\|_{\mathcal{B}_{\alpha}^{2}}\right),
\end{equation}
\begin{equation}\label{contraction:4}
\left\|I^{-}\left(\mathbb{P}\Big(\Upsilon(\cdot,\cdot,Dw_{1},\mu_{1})-\Upsilon(\cdot,\cdot(Dw_{2},\mu_{2}))\Big)\right)
\right\|_{\mathcal{B}_{\alpha}^{2}}\leq\frac{1}{7}\left(\|w_{1}-w_{2}\|_{\mathcal{B}_{\alpha}^{2}}
+\|\mu_{1}-\mu_{2}\|_{\mathcal{B}_{\alpha}^{2}}\right),
\end{equation}
\begin{multline}\label{contraction:5}
\left\|I^{-}\left(\mathbb{P}\left(bI^{+}\left(\mathrm{div}\Big(\mu_{1}\Theta(\cdot,\cdot,Dw_{1},\mu_{1})
-\mu_{2}\Theta(\cdot,\cdot,Dw_{2},\mu_{2})\Big)\right)\right)\right)\right\|_{\mathcal{B}_{\alpha}^{2}}
\\
\leq\frac{1}{7}\left(\|w_{1}-w_{2}\|_{\mathcal{B}_{\alpha}^{2}}
+\|\mu_{1}-\mu_{2}\|_{\mathcal{B}_{\alpha}^{2}}\right),
\end{multline}
\begin{multline}\label{contraction:6}
\bar{m}\left\|I^{-}\left(\mathbb{P}\left(bI^{+}\left(\mathrm{div}\Big(\Theta(\cdot,\cdot,Dw_{1},\mu_{1})
-\Theta(\cdot,\cdot,Dw_{2},\mu_{2})\Big)\right)\right)\right)\right\|_{\mathcal{B}_{\alpha}^{2}}
\\
\leq\frac{1}{7}\left(\|w_{1}-w_{2}\|_{\mathcal{B}_{\alpha}^{2}}
+\|\mu_{1}-\mu_{2}\|_{\mathcal{B}_{\alpha}^{2}}\right).
\end{multline}
If we are able to establish the estimates \eqref{contraction:1}--\eqref{contraction:6}, then we will have succeeded
in demonstrating \eqref{contraction:toShow} with the constant $\lambda=6/7.$
We omit the remaining details, but we mention that the estimates \eqref{contraction:1}--\eqref{contraction:6} follow
from the assumptions {\bf(A1)} and {\bf(A2)}, using the other tools we have used previously, such as the triangle inequality,
the algebra property of $\mathcal{B}_{\alpha}^{2},$ the mapping properties of $I^{+}$ and $I^{-},$ and so on.

Having established that $\mathcal{T}$ is a contraction on $X,$ we are guaranteed the existence
of $(w,\mu)\in\mathcal{B}_{\alpha}^{2}\times\mathcal{B}_{\alpha}^{2}$ which satisfies
the Duhamel formulation of the payoff problem, \eqref{duhamelMu} and \eqref{duhamelWLast}.  
We may work backward and find that, in fact, \eqref{duhamelWFirst} is also satisfied.
As we have discussed in Section \ref{section:Wiener}, the finiteness of the $\mathcal{B}_{\alpha}^{2}$ norm implies that 
$w$ and $\mu$ are analytic for each $t\in(0,T).$  This is certainly enough regularity to be able to differentiate
\eqref{duhamelMu} and 
\eqref{duhamelWFirst} with respect to time; the result of this operation is the conclusion that $w$ and $\mu$
are strong solutions of \eqref{muEvolution}, \eqref{muBC}, \eqref{wEvolution}, \eqref{wBC}.  
We define $m=\mu+\bar{m},$ and see that
$m$ satisfies \eqref{mEquation}.  Finally, we must solve for the mean of $u.$  We note that from \eqref{uEquation},  
using $Du=Dw,$ the
mean of $u$ satisfies $(1-\mathbb{P})u_{t}=-(1-\mathbb{P})\mathcal{H}(\cdot,\cdot,Dw,m).$  Since the right-hand side
of this equation
has been determined, and since we may compute the terminal mean of $u$ from \eqref{BCs}, we see that we can
integrate in time to find the mean of $u.$  Adding the mean of $u$ to $w,$ we find $u$ which satisfies \eqref{uEquation}.
This completes the proof of the theorem.
\end{proof}

\subsection{The planning problem}\label{section:planning}

The above formulation and existence theorem can be readily adapted to the so-called planning problem, in which
\eqref{BCs} is replaced by
\begin{equation}\label{BCs2}
m(0,x)=m_{0}(x),\qquad u(T,x)=u_{T}(x).
\end{equation}
For our method, the planning problem is slightly more straightforward to treat.  The only modification in the formulation
is that \eqref{duhamelWLast} is replaced with 
\begin{multline}\nonumber
w(t,\cdot)=e^{\Delta(T-t)}w_{T}-I^{-}(\mathbb{P}\Upsilon(\cdot,\cdot,Dw,\mu))(t)
-I^{-}(\mathbb{P}(be^{\Delta \cdot}\mu_{0}))(t)\\
-I^{-}(\mathbb{P}(bI^{+}\mathrm{div}(\mu\Theta(\cdot,\cdot,Dw,\mu))(\cdot)))(t)
-\bar{m}I^{-}(\mathbb{P}(bI^{+}\mathrm{div}(\Theta(\cdot,\cdot,Dw,\mu))(\cdot)))(t).
\end{multline}
Of course, $w_{T}=\mathbb{P}u_{T}.$

We arrive at the second of our main theorems, which is analagous to Theorem \ref{smallDataTheorem1}, but is now
for the planning problem.  Note that since assumption {\bf(A1)} concerned the payoff function, $G,$ it is now irrelevant;
we continue to use assumption {\bf(A2)}.
\begin{theorem}\label{smallDataTheorem2}
Let $T>0$ and $\alpha\in(0,T/2)$ be given.  Let assumption {\bf(A2)} be satisfied, and assume 
$b\in\mathcal{B}_{\alpha}^{0}.$
There exists $\delta>0$ such that if $u_{T}$ and the probability measure $m_{0}$  are such that 
$w_{T}=\mathbb{P}u_{T}$ and $\mu_{0}=m_{0}-\bar{m}$ 
satisfy $|w_{T}|_{B^{2}}+|\mu_{0}|_{B^{2}}<\delta,$ then the system \eqref{uEquation}, \eqref{mEquation}, 
\eqref{BCs2} has a strong, locally unique 
solution $(u,m)\in(\mathcal{B}_{\alpha}^{2})^{2}.$  Furthermore, for all $t\in(0,T),$ each of $u(t,\cdot)$ and $m(t,\cdot)$ are 
analytic, and $m(t,\cdot)$ is a probability measure.
\end{theorem}

We omit the proof, since it is almost identical to the proof of Theorem \ref{smallDataTheorem1}.  We note
that the primary modification is that the definition of the center of the ball $X,$ given in \eqref{definitionOfBallX},
is changed to
$(a_{0},b_{0})=(e^{\Delta t}\mu_{0}, e^{\Delta(T-t)}w_{T}+\tau I^{-}(e^{\Delta\cdot}\mu_{0})(t)).$

\subsection{Examples}\label{exampleSection}

In \cite{ambroseMFG}, we gave examples with a separable Hamiltonian which satisfy our assumptions.
Included there, in the terminology of the present work, was the example 
$\mathcal{H}(t,x,p,m)=a(t,x)|p|^{4}+m^{3},$ with $a\in\mathcal{B}_{\alpha}^{2}.$  We also included the
example $\mathcal{H}(t,x,p,m)=a(t,x)p_{i}p_{j}p_{k}+m^{3},$ with $i,$ $j,$ $k$ each in $\{1,2,\ldots,n\},$ and
again with $a\in\mathcal{B}_{\alpha}^{2}.$

We have now demonstrated, however, that we need not consider only separable examples.  So, for instance,
we can generalize the above and verify that our assumptions are satisfied for 
$\mathcal{H}(t,x,p,m)=a_{1}(t,x)p_{i}p_{j}p_{k}m^{\ell}+a_{2}(t,x)m^{\sigma},$ and also for
$\mathcal{H}(t,x,p,m)=a_{1}(t,x)|p|^{4}m^{\ell}+a_{2}(t,x)m^{\sigma},$
with $\ell$ and $\sigma$ being
natural numbers, and with $a_{1}$ and $a_{2}$ each being elements of $\mathcal{B}_{\alpha}^{2}.$

To expand upon one example somewhat, we let $\mathcal{H}(t,x,Du,m)=m^{2}|Du|^{4}+m^{3}.$  Then,
\begin{equation}\nonumber
\Xi=\mathbb{P}\left((\mu+\bar{m})^{2}|Dw|^{4}\right)+\mathbb{P}\left((\mu+\bar{m})^{3}\right).
\end{equation}
We can write $\mathbb{P}\left((\mu+\bar{m})^{3}\right)=\mathbb{P}\left(\mu^{3}+3\bar{m}\mu^{2}+3\bar{m}^{2}\mu\right).$
Thus, we take the following:
\begin{equation}\nonumber
b=3\bar{m}^{2},\qquad \Upsilon=\mathbb{P}\left((\mu+\bar{m})^{2}|Dw|^{4}+\mu^{3}+3\bar{m}\mu^{2}\right).
\end{equation}
We could also compute $\Theta,$ and we may conclude that $b,$ $\Upsilon,$ and $\Theta$  satisfy the given conditions.

\section{Large data, small Hamiltonians}\label{largeDataSection}

We now consider the case of weak coupling between the evolution equations, replacing $\mathcal{H}$
with $\varepsilon\mathcal{H},$ with $\varepsilon$ representing a small parameter.  
Thus, for the present section, the equations under consideration are:
\begin{equation}\label{epsilon1}
u_{t}+\Delta u + \varepsilon\mathcal{H}(t,x,Du,m)=0,
\end{equation}
\begin{equation}\label{epsilon2}
m_{t}-\Delta m  + \varepsilon\mathrm{div}(m\mathcal{H}_{p}(t,x,Du,m))=0.
\end{equation}
We may repeat the
previous steps in writing our Duhamel formulation, although for the present purpose, we may end the process
sooner.
We restate \eqref{duhamelMu}:
\begin{equation}\label{duhamelMu-restated}
\mu(t,\cdot)=e^{\Delta t}\mu_{0}
+\varepsilon I^{+}(\mathrm{div}((\mu+\bar{m}) \Theta(\cdot,\cdot,Dw,\mu)))(t,\cdot).
\end{equation}
We also restate \eqref{duhamelWFirst}, but with two modifications: we change the data to reflect the planning problem boundary
conditions, and we use $\Xi$ instead of $\Upsilon$ and $b\mu.$  These considerations yield the following version of
our Duhamel equation for $w:$
\begin{equation}\label{duhamelW-restated}
w(t,\cdot)=
e^{\Delta(T-t)}w_{T}-\varepsilon I^{-}(\Xi(\cdot,\cdot,Dw,\mu))(t).
\end{equation}

Based on these equations, we define the mapping $F$ as follows:
\begin{equation}\nonumber
F\left(\left(\begin{array}{c}w\\ \mu \end{array}\right),\varepsilon\right)=\left(\begin{array}{c}
w-e^{\Delta(T-\cdot)}w_{T}+\varepsilon I^{-}(\Xi(\cdot,\cdot,Dw,\mu))\\
\mu-e^{\Delta\cdot}\mu_{0}-\varepsilon I^{+}(\mathrm{div}((\mu+\bar{m})\Theta(\cdot,\cdot,Dw,\mu)))
\end{array}\right).
\end{equation}
When $\varepsilon=0,$ we know a solution of $F((w,\mu),0)=0;$ this is simply $w(t,\cdot)=e^{\Delta(T-t)}w_{T},$ and
$\mu(t,\cdot,)=e^{\Delta t}\mu_{0}.$  Computation of the derivative of $F$ is straightforward:
\begin{equation}\nonumber
D_{(w,\mu)}F\Big|_{\varepsilon=0}=\mathrm{Id}.
\end{equation}
The identity map is, naturally, a bijection, so the implicit function theorem applies.

There are of course many statements of the implicit function theorem; the following version may be found
in \cite{hunter}.
\begin{theorem}[Implicit Function Theorem] 
Let $X,$ $Y,$ and $Z$ be Banach spaces.  Let $U$ be an open subset of $X\times Y,$ and suppose 
$F:U\rightarrow Z$ is a continuously differentiable map.  Let $(x_{0},y_{0})\in U$ be such that $F(x_{0},y_{0})=0.$
If $D_{y}F(x_{0},y_{0}):Y\rightarrow Z$ is a one-to-one, onto, bounded linear map, then there exists $V\subseteq X,$ an 
open neighborhood of $x_{0},$ and there exists $W\subseteq Y,$ an open neighborhood of $y_{0},$ and a unique
continuously differentiable function $f:V\rightarrow W$ such that $F(x,f(x))=0,$ for all $x\in V.$
\end{theorem}

With an eye towards using the implicit function theorem, we now state a new assumption on the Hamiltonian, 
$\mathcal{H}.$

\noindent
{\bf (A3)} $\mathcal{H}$ is such that 
$\Xi:(\mathcal{B}_{\alpha}^{1})^{n}\times\mathcal{B}_{\alpha}^{2}\rightarrow\mathcal{B}_{\alpha}^{0}$ is continuously
differentiable, and such that 
$\Theta:(\mathcal{B}_{\alpha}^{1})^{n}\times\mathcal{B}_{\alpha}^{2}\rightarrow(\mathcal{B}_{\alpha}^{1})^{n}$ is
continuously differentiable.

With this assumption in hand, and with the previously developed mapping properties of $I^{+}$ and $I^{-},$
we see that for $w_{T}\in B^{2}$ and $\mu_{0}\in B^{2},$
$F$ maps $\mathcal{B}_{\alpha}^{2}\times\mathcal{B}_{\alpha}^{2}\times\mathbb{R}$ 
into $\mathcal{B}_{\alpha}^{2}\times\mathcal{B}_{\alpha}^{2}.$  We use the implicit function theorem with $X=\mathbb{R}$
and $Y=Z=\mathcal{B}_{\alpha}^{2}\times\mathcal{B}_{\alpha}^{2}.$
We have proved the following theorem.

\begin{theorem}\label{useOfIFTForMFG} Let $T>0$ and $\alpha\in(0,T/2)$ be given.  Let $m_{0}\in B^{2}$ be a probability measure and
let $u_{T}\in B^{2}.$  Assume that $\mathcal{H}$ satisfies {\bf(A3)}.  Then there exists $\varepsilon_{0}>0$ such that
for all $\varepsilon\in(-\varepsilon_{0},\varepsilon_{0}),$ there exist unique
$u(\cdot,\cdot;\varepsilon)\in\mathcal{B}_{\alpha}^{2}$
and $m(\cdot,\cdot;\varepsilon)\in\mathcal{B}_{\alpha}^{2}$ which solve \eqref{epsilon1}, \eqref{epsilon2}, \eqref{BCs2}.
At any time $t\in(0,T/2),$ each of $u$ and $m$ are analytic, and $m$ is a probability measure.
\end{theorem}

We make a few remarks on the theorem: because $m_{0}$ is a probability measure and $m$ is a strong solution of
\eqref{epsilon2}, 
we conclude that $m$ is a probability measure at positive times as well.  We have discussed convexity above, and
we note now that even if $\mathcal{H}$ were convex, $-\mathcal{H}$ would not be; since the implicit function theorem
works with both positive and negative values of $\varepsilon,$ this clearly shows that the present method is entirely different
from methods relying upon convexity.  Along the same lines, we mention that in Section \ref{section:ellipticity} above, we discussed
ensuring that the system is elliptic.  In the present case, the ellipticity is ensured by taking $\varepsilon$ sufficiently small,
so that the linear elliptic terms are dominant.

The details of the proof of Theorem \ref{useOfIFTForMFG}, aside from the use of the implicit function theorem, are
the same as the proof of Theorem \ref{smallDataTheorem1}.  
That is, the implicit function theorem guarantees the existence of $(w,\mu).$
The finiteness of the $\mathcal{B}_{\alpha}^{2}$ norm ensures analyticity, and thus these are strong solutions.
We may then use the evolution equation for $u$ to recover the mean of $u$ at each time.  

\subsection{Examples}

All of the examples of Section \ref{exampleSection} are also valid here.  However, further examples are now available, as 
{\bf(A3)} is less stringent than {\bf(A2)}.  We mention that as an additional example to those of Section 
\ref{exampleSection}, we could now consider, for instance, $\mathcal{H}(t,x,Du,m)=m^{j}|Du|^{2},$ for $j\in\mathbb{N}.$

\section*{Acknowledgments}
The author gratefully acknowledges support from the National Science Foundation through 
grant  DMS-1515849.  The author is grateful to Benjamin Moll for helpful correspondence.
\bibliography{ambroseMeanField}{}
\bibliographystyle{plain}

\end{document}